\documentclass{amsart}
\usepackage{amsmath,amscd}
\usepackage{amsfonts,amssymb}
\usepackage[all]{xy}
\usepackage{enumerate}
\numberwithin{equation}{subsection}
\theoremstyle{plain}
\newtheorem{thm}[subsection]{Theorem}
\newtheorem{prop}[subsection]{Proposition}
\newtheorem{lemma}[subsection]{Lemma}
\newtheorem{cor}[subsection]{Corollary}
\theoremstyle{definition}
\newtheorem{defn}[subsection]{Definition}
\newtheorem{notn}[subsection]{Notation}
\newtheorem{cont}[subsection]{Contents}
\newtheorem{ackn}[subsection]{Acknowledgement}
\theoremstyle{remark}

\begin{document}
\title{Hodge theory and the Mordell-Weil rank of elliptic curves over extensions of function fields}
\author{Ambrus P\'al}
\date{November 8, 2013.}
\address{Department of Mathematics, 180 Queen's Gate, Imperial College, London, SW7 2AZ, United Kingdom}
\email{a.pal@imperial.ac.uk}
\begin{abstract} We use Hodge theory to prove a new upper bound on the ranks of Mordell-Weil groups for elliptic curves over function fields after regular geometrically Galois extensions of the base field, improving on previous results of Silverman and Ellenberg, when the base field has characteristic zero and the supports of the conductor of the elliptic curve and of the ramification divisor of the extension are disjoint.
\end{abstract}
\maketitle
\pagestyle{myheadings}
\markright{Bounds on the Mordell-Weil rank of elliptic surfaces}

\section{Introduction}

For every field $F$ let $\overline F$ denote its separable closure. For every field extension $K|F$ and variety $Z$ defined over $F$ let $Z_K$ denote the base change of $Z$ to $K$. Let $\mathcal C$ be a smooth projective geometrically irreducible curve defined over a field $k$ and let $F$, $g$ denote its function field and its genus, respectively. Let $\pi:\mathcal C'\rightarrow\mathcal C$ be a finite regular geometrically Galois cover defined over $k$. Let $F'$ be the function field of $\mathcal C'$ and let $S$ is the reduced ramification divisor of the cover $\pi$. Let $G=\textrm{Aut}(\mathcal C'_{\overline k}|\mathcal C_{\overline k})$ and let $\Sigma$ be the image of $\textrm{Gal}(\overline k|k)$ in Aut$(G)$ with respect to 
the natural action of $\textrm{Gal}(\overline k|k)$ on $G$. 

Next we recall the definition of Ellenberg's constant. Let $V$ be the real vector space spanned by the irreducible complex-valued characters of $G\rtimes\Sigma$, and let $W$ be the real vector space spanned by the irreducible complex-valued characters of $G$. We say that a vector $v$ in $V$ (resp. in $W$) is non-negative if its inner product with each irreducible representation of $G\rtimes\Sigma$ (resp. of $G$) is non-negative. Let $c\in V$ be the coset character of $G\rtimes\Sigma$ attached to $\Sigma$, and let $r\in W$ be the regular character of $G$. Ellenberg defines the constant $\epsilon(G,\Sigma)$ as the maximum of the inner product $\langle v,c\rangle$ over all $v\in V$ such that
\begin{enumerate}
\item[$(i)$] $v$ is non-negative;
\item[$(ii)$] $r-R(v)$ is non-negative, where $R:V\rightarrow W$ is the restriction map.
\end{enumerate}
The region of $V$ defined by these two conditions above is a compact polytope, so $\epsilon(G,\Sigma)$ is well-defined. 

Let $E$ be a non-isotrivial elliptic curve over $F$. By the Lang--N\'éron theorem (see \cite{Co}) the group $E(F')$ is finitely generated. Let $c_E$ denote the degree of the conductor of $E$ and let $d_E$ denote the degree of the minimal discriminant of $E$. Our main result is the following
\begin{thm} Assume that $k$ has characteristic zero, and the supports of $S$ and of the conductor of $E$ are disjoint. Then
\begin{equation}
\textrm{\rm rank}(E(F'))\leq\epsilon(G,\Sigma)(c_E-d_E/6+2g-2+\deg(S)).
\end{equation}
\end{thm}
Let $O(G,\Sigma)$ be the cardinality of the set of orbits of $G$ with respect to the action of $\Sigma$. It is easy to prove that $O(G,\Sigma)=\epsilon(G,\Sigma)$ when $G$ is an abelian group (see Proposition 2.11 of \cite{El}). Hence we have the following immediate
\begin{cor} Assume that $k$ has characteristic zero, the supports of $S$ and of the conductor of $E$ are disjoint, and $G$ is abelian. Then
\begin{equation}
\textrm{\rm rank}(E(F'))\leq O(G,\Sigma)\big(c_E-d_E/6+2g-2+\deg(S)).
\end{equation}
\end{cor}
The upper bound:
\begin{equation}
\textrm{\rm rank}(E(F'))\leq O(G,\Sigma)(c_E+4g-4),
\end{equation}
was first proved by Silverman in \cite{Si} under the assumption that $G$ is abelian, the cover $\pi$ is unramified (i.e$\textrm{.}$ $S$ is empty), $k$ is a number field, and a weak form of Tate's conjecture holds for $\mathcal E'$ where $g':\mathcal E'\rightarrow\mathcal C'$ is the unique relatively minimal elliptic surface over $k$ whose generic fibre is $E_{F'}$.  Later in \cite{El} Ellenberg proved the following more general unconditional bound:
\begin{equation}
\textrm{\rm rank}(E(F'))\leq\epsilon(G,\Sigma)(c_E+4g-4+2\deg(S)),
\end{equation}
without any restriction on the group, assuming that the characteristic of $k$ is at least $5$, and without assuming that supports of $S$ and of the conductor of $E$ are disjoint. Note that (1.1.1) trivially implies (1.2.3) since $d_E\geq12$ for every non-isotrivial elliptic curve $E$. In fact $c_E\leq d_E$, so we also have the following immediate consequence of the main result: 
$$
\textrm{\rm rank}(E(F'))\leq\epsilon(G,\Sigma)(5c_E/6+2g-2+\deg(S)),$$
and hence our result is stronger than Ellenberg's bound. Moreover the bound (1.1.1) is false for fields of positive characteristic, even in the degenerate case when $G$ is trivial, see for example \cite{Ul}.

The strategy for proving Theorem 1.1 follows Ellenberg's original idea; it is enough to show that the $\mathbb C[G]$-module $E(K)\otimes\mathbb C$ is the quotient of the free $\mathbb C[G]$-module of rank $c_E-d_E/6+2g-2+\deg(S)$, where $K$ is the composition of the fields $F'$ and $\overline k$. In the course of our proof of this fact we may assume without the loss of generality that $k$ is algebraically closed. In fact by the Lefschetz priciple we may also assume that $k$ is the field of complex numbers, which we will do from now on. However our method is different because we use de Rham cohomology instead of \'etale cohomology and the easy direction of the Lefschetz-$(1,1)$ theorem. (A similar idea is used in \cite{Fa} by Fastenberg.) We prove the required bounds on the multiplicity of irreducible representations of $G$ appearing in certain cohomology groups of $\mathcal E'$ via a simple equivariant Riemann--Roch theorem (see Theorem 2.5) and an equivariant Grothendieck--Ogg--Shafarevich formula for the Euler characteristic of constructible sheaves of complex vector spaces (Theorem 3.6) and its applications to elliptic surfaces (Theorems 3.7 and 4.4) which we think are interesting results on their own.
\begin{ackn} The author was partially supported by the EPSRC grants P19164 and P36794.
\end{ackn}
\begin{cont} In the next chapter we compute the $\mathbb C[G]$-module structure of the cohomology group $H^2(\mathcal E',\mathcal O_{\mathcal E'})$ using essentially the same arguments as in my previous paper \cite{AP}. We use the Grothendieck--Ogg--Shafarevich formula to compute the $\mathbb C[G]$-module structure of the cohomology group $H^1(\mathcal C',R^1g'_*\mathbb C)$ in the third chapter. In the last section we combine these results and the easy direction of the Lefschetz-$(1,1)$ theorem to conclude the proof of our main result.
\end{cont}

\section{Equivariant Riemann--Roch for elliptic surfaces}

\begin{defn} In this chapter $G$ will denote a finite group. Let $X$ be a normal scheme which is of finite type over Spec$(\mathbb C)$. Assume that the finite group $G$ acts on $X$ on the left. Let $\mathcal F$ be a coherent sheaf on $X$. A $G$-linearization on $\mathcal F$ is a collection $\Psi=\{\psi_g\}_{g\in G}$ of isomorphisms $\psi_g:g_*(\mathcal F)\rightarrow\mathcal F$ for every $g\in G$ such that
\begin{enumerate}
\item[$(i)$] we have $\psi_1=\text{Id}_{\mathcal F}$,
\item[$(ii)$] for every $g$, $h\in G$ we have $\psi_{hg}=\psi_h\circ
h_*(\psi_g)$,
\end{enumerate}
where $h_*(\psi_g):(hg)_*(\mathcal F)=h_*(g_*(\mathcal F))\rightarrow h_*(\mathcal F)$ is the direct image of the map $\psi_g:g_*(\mathcal F)\rightarrow\mathcal F$ under the action of $h$. We define a $G$-sheaf over $X$ to be a sheaf on $X$ equipped with a $G$-linearisation. A coherent $G$-sheaf is a coherent sheaf on $X$ equipped with a $G$-linearisation $\Psi$ such that $\psi_g:g_*(\mathcal F)\rightarrow\mathcal F$ is $\mathcal O_X$-linear for every $g\in G$. 
\end{defn}
\begin{defn} Let $K(\mathbb C[G])$ denote the Grothendieck group of all finitely generated $\mathbb C[G]$-modules. For every finitely $\mathbb C[G]$-module $M$ let $[M]$ denote its class in $K(\mathbb C[G])$. Let $f:X\rightarrow Y$ be a $G$-cover and let $\mathcal F$ be a coherent $G$-sheaf on $X$. Also assume that $Y$ is proper over $\text{\rm Spec}(\mathbb C)$; then for every $n\in\mathbb N$ the cohomology group $H^n(X,\mathcal F)$ is a finitely generated $\mathbb C[G]$-module with respect to the natural $\mathbb C[G]$-action. Therefore the element:
$$\chi_G(X,\mathcal F)=\sum_{n\in\mathbb N}(-1)^n
[H^n(X,\mathcal F)]\in K(\mathbb C[G])$$
is well-defined.
\end{defn}
In addition to the assumptions above also suppose now that $f:X\rightarrow Y$ above is a map of smooth, projective curves over Spec$(\mathbb C)$. Let $\mathcal L$ be a line bundle on $Y$. The line bundle $f^*(\mathcal L)$ on $X$ is naturally equipped with the structure of a coherent $G$-sheaf. 
\begin{lemma} With the same notation and assumptions as above the following equation holds in $K(\mathbb C[G])$:
$$\chi_G(X,f^*(\mathcal L))=\chi_G(X,\mathcal O_X)+\deg(\mathcal L)[\mathbb C[G]].$$
\end{lemma}
\begin{proof} This is the content of Lemma 5.6 of \cite{AP} on page 524.
\end{proof}
\begin{notn} Suppose now that the map $f:X\rightarrow Y$ is the cover $\pi:\mathcal C'\rightarrow\mathcal C$ of the introduction and let $\Delta_E$ denote the discriminant of a relatively minimal elliptic surface $g:\mathcal E\rightarrow\mathcal C$ whose generic fibre is $E$. Assume that the ramification divisor of the cover $\pi$ has support disjoint from the conductor of $E$. Let $g':\mathcal E'\rightarrow\mathcal C'$ be the base change of the elliptic fibration $g:\mathcal E\rightarrow\mathcal C$ with respect to the map $\pi$. Note that the $\mathcal C'$-scheme $\mathcal E'$ is a relatively minimal regular model of the base change of $E$ to the function field of $\mathcal C'$ because $f$ does not ramify at the locus of the conductor of $E$. Moreover $\mathcal E'$ is equipped with a unique action of $G$ such that $g'$ is equivariant with respect to this action and the one on $\mathcal C'$.  
\end{notn}
\begin{thm} We have: 
$$[H^2(\mathcal E',\mathcal O_{\mathcal E'})]=
\frac{d_E}{12}[\mathbb C[G]]-\chi_G(\mathcal C',\mathcal O_{\mathcal C'}).$$
\end{thm}
\begin{proof} By definition we have:
\begin{equation}
\chi_G(\mathcal E',\mathcal O_{\mathcal E'})=[H^0(\mathcal E',\mathcal O_{\mathcal E'})]-[H^1(\mathcal E',\mathcal O_{\mathcal E'})]+[H^2(\mathcal E',\mathcal O_{\mathcal E'})].
\end{equation}
By Lemma 4 of \cite{GSz} on page 79 the map $(g')^*:\text{Pic}(\mathcal C')\rightarrow\text{Pic}(\mathcal E')$ induced by Picard functoriality is an isomorphism. Because this map is equivariant with respect to the induced $G$-actions on Pic$(\mathcal C')$ and Pic$(\mathcal E')$, we get that $H^1(\mathcal E',\mathcal O_{\mathcal E'})\cong H^1(\mathcal C',\mathcal O_{\mathcal C'})$ as $\mathbb C[G]$-modules, since these modules are isomorphic to the tangent spaces at the zero of the abelian varieties $\text{Pic}(\mathcal E')$ and $\text{Pic}(\mathcal C')$, respectively. Obviously
$H^0(\mathcal E',\mathcal O_{\mathcal E'})\cong H^0(\mathcal C',\mathcal O_{\mathcal C'})$ as $\mathbb C[G]$-modules, hence from (2.5.1) we get that
\begin{equation}
[H^2(\mathcal E',\mathcal O_{\mathcal E'})]=\chi_G(\mathcal E',\mathcal O_{\mathcal E'})-\chi_G(\mathcal C',\mathcal O_{\mathcal C'}).
\end{equation}
Let $\Omega^1_{\mathcal E/\mathcal C}$, $\Omega^1_{\mathcal E'/\mathcal C'}$ denote the sheaf of relative K\"ahler differentials of the $\mathcal C$-scheme $\mathcal E$ and the $\mathcal C'$-scheme $\mathcal E'$, respectively. Let $\omega_{\mathcal E/\mathcal C}$, $\omega_{\mathcal E'/\mathcal C'}$ denote respectively the pull-back of $\Omega^1_{\mathcal E/\mathcal C}$ and $\Omega^1_{\mathcal E'/\mathcal C'}$ with respect to the zero section. These sheaves are line bundles on $\mathcal C$ and $\mathcal C'$, respectively. Moreover by Grothendieck's duality we have $R^1g'_*(\mathcal O_{\mathcal E'})=\omega^{\otimes-1}_{\mathcal E'/\mathcal C'}$. Because all boundary maps in the spectral sequence $H^p(\mathcal C',R^qg'_*(\mathcal O_{\mathcal E'}))
\Rightarrow H^{p+q}(\mathcal E',\mathcal O_{\mathcal E'})$ are $\mathbb C[G]$-linear we get from the above that
\begin{eqnarray}
\chi_G(\mathcal E',\mathcal O_{\mathcal E'})&=&
\chi_G(\mathcal C',\mathcal O_{\mathcal C'})-
\chi_G(\mathcal C',R^1g'_*(\mathcal O_{\mathcal E'}))\\
&=&
\chi_G(\mathcal C',\mathcal O_{\mathcal C'})-
\chi_G(\mathcal C',\omega^{\otimes-1}_{\mathcal E'/\mathcal C'}).
\nonumber
\end{eqnarray}
Combining (2.5.2) and (2.5.3) we get that
\begin{equation}
[H^2(\mathcal E',\mathcal O_{\mathcal E'})]=-\chi_G(\mathcal C',\omega^{\otimes-1}_{\mathcal E'/\mathcal C'}).
\end{equation}
By definition $\Delta_E$ is the zero divisor of a non-zero section of $\omega_{\mathcal E/\mathcal C}^{ \otimes12}$. Therefore $\deg(\Delta_E)=12\deg(\omega_{\mathcal E/\mathcal C})$. Moreover $\omega_{\mathcal E'/\mathcal C'}=\pi^*(\omega_{\mathcal E/\mathcal C})$. Hence (2.5.4) and Lemma 2.3 imply that
$$[H^2(\mathcal E',\mathcal O_{\mathcal E'})]=
\frac{\deg(\Delta_E)}{12}[\mathbb C[G]]-\chi_G(\mathcal C',\mathcal O_{\mathcal C'}).$$
\end{proof}

\section{Equivariant Grothendieck--Ogg--Shafarevich for elliptic surfaces}

\begin{notn} Let $X$ be a quasi-projective variety over $\mathbb C$. For every constructible sheaf $\mathcal F$ of complex vector spaces and for every $n\in\mathbb N$ let $H^n(X,\mathcal F)$ denote the $n$-th cohomology group of $\mathcal F$ on $X$ with respect to the analytical topology and let
$$\chi(X,\mathcal F)=\sum_{n\in\mathbb N}(-1)^n\dim_{\mathbb C}\big(
H^n(X,\mathcal F)\big)\in\mathbb N$$
denote the Euler-characteristic of $\mathcal F$. Assume now that $X$ be a smooth irreducible projective curve and let $\mathcal F$ be a constructible sheaf of complex vector spaces on $X$. There is an open subcurve $U\subseteq X$ such that $\mathcal F|_U$ is locally free of finite rank. We define $\textrm{\rm rank}(\mathcal F)$, the rank of $\mathcal F$, as the rank of $\mathcal F|_U$. For every $x\in X(\mathbb C)$ let $\mathcal F_x$ and $c_x(\mathcal F)$ denote the stalk of $\mathcal F$ at $x$ and the 
conductor of $\mathcal F$ at $x$ given by the formula:
$$c_x(\mathcal F)=\textrm{\rm rank}(\mathcal F)-
\dim_{\mathbb C}(\mathcal F_x),$$
respectively. Finally let
$$\textrm{\rm cond}(\mathcal F)=
\sum_{x\in X(\mathbb C)}c_x(\mathcal F).$$
\end{notn}
In this section our main tool will be the complex analytic version of the Grothen\-dieck--Ogg--Shafarevich formula:
\begin{thm}[Grothendieck--Ogg--Shafarevich] For every $X$ and $\mathcal F$ as above we have:
$$\chi(X,\mathcal F)=\textrm{\rm rank}(\mathcal F)\chi(X,\mathbb C)-
\textrm{\rm cond}(\mathcal F).$$
\end{thm} 
\begin{proof} The theorem could be easily reduced to the Grothen\-dieck--Ogg--Shafarevich formula for torsion constructible sheaves (see Th\'eor\`eme 1 of \cite{Ra} on page 133) using a standard set of arguments. We omit the details.
\end{proof}
\begin{lemma} Let $\mathcal F$ and $\mathcal G$ be constructible sheaves of complex vector spaces on $X$. Assume that $\mathcal G$ is locally free in a neighbourhood of $x\in X(\mathbb C)$. Then
$$c_x(\mathcal F\otimes\mathcal G)=\textrm{\rm rank}(\mathcal G)c_x(\mathcal F).$$
\end{lemma}
\begin{proof} Because $\mathcal G$ is locally free in a neighbourhood of $x\in X(\mathbb C)$, we have:
$$(\mathcal F\otimes\mathcal G)_x=\mathcal F_x\otimes\mathcal G_x
\textrm{\rm, and hence }\dim_{\mathbb C}
\big((\mathcal F\otimes\mathcal G)_x\big)=\textrm{\rm rank}(\mathcal G)
\dim_{\mathbb C}(\mathcal F_x).$$
Since
$$\textrm{\rm rank}(\mathcal F\otimes\mathcal G)=
\textrm{\rm rank}(\mathcal G)\textrm{\rm rank}(\mathcal F),$$
the claim follows immediately.
\end{proof}
\begin{defn} Let $G$ continue to denote a finite group. Let $X$ be a quasi-projective variety over $\mathbb C$ and assume that $G$ acts on $X$ on the left. Let $\mathcal F$ be a constructible sheaf of complex vector spaces on $X$. A $G$-linearisation on $\mathcal F$ is a collection $\Psi=\{\psi_g\}_{g\in G}$ of $\mathbb C$-linear isomorphisms of sheaves $\psi_g:g_*(\mathcal F)\rightarrow\mathcal F$ for every $g\in G$ such that
\begin{enumerate}
\item[$(i)$] we have $\psi_1=\text{Id}_{\mathcal F}$,
\item[$(ii)$] for every $g$, $h\in G$ we have $\psi_{hg}=\psi_h\circ
h_*(\psi_g)$,
\end{enumerate}
where $h_*(\psi_g):(hg)_*(\mathcal F)=h_*(g_*(\mathcal F))\rightarrow h_*(\mathcal F)$ is the direct image of the map $\psi_g:g_*(\mathcal F)\rightarrow\mathcal F$ under the action of $h$. We define a complex constructible $G$-sheaf over $X$ to be a constructible sheaf of complex vector spaces on $X$ equipped with a $G$-linearisation. 
\end{defn}
\begin{defn} Let $X$ be a quasi-projective variety over $\mathbb C$. Let $f:X\rightarrow Y$ be a $G$-cover and let $\mathcal F$ be a complex constructible $G$-sheaf on $X$. Then for every $n\in\mathbb N$ the cohomology group $H^n(X,\mathcal F)$ is a finitely generated $\mathbb C[G]$-module with respect to the natural $\mathbb C[G]$-action. Therefore the element:
$$\chi_G(X,\mathcal F)=\sum_{n\in\mathbb N}(-1)^n
[H^n(X,\mathcal F)]\in K(\mathbb C[G])$$
is well-defined.
\end{defn}
In addition to the assumptions above also suppose now that $f:X\rightarrow Y$ above is a map of smooth, projective curves over Spec$(\mathbb C)$. Let $\mathcal F$ be a constructible sheaf of complex vector spaces on $Y$. Then $f^*(\mathcal F)$ is naturally equipped with the structure of a complex constructible $G$-sheaf on $X$. Let $j:U\rightarrow Y$ be an open immersion and assume that $\mathcal F|_U$ is a locally free of finite rank. Assume that the map $f$ does not ramify on the complement of $U$ in $Y$. 
\begin{thm} With the same notation and assumptions as above the following equation holds in $K(\mathbb C[G])$:
$$\chi_G(X,f^*(\mathcal F))=\textrm{\rm rank}(\mathcal F)\chi_G(X,\mathbb C)-
\textrm{\rm cond}(\mathcal F)[\mathbb C[G]].$$
\end{thm}
\begin{proof} Let $R(G)$ denote the set of irreducible $\mathbb C$-valued characters of $G$. For every $\alpha\in R(G)$ let $\Delta_{\alpha}:K(\mathbb C[G])\rightarrow\mathbb Z$ denote the unique homomorphism such that $\Delta_{\alpha}([M])$ is the multiplicity of $\alpha$ in the character of $M$ for every finitely generated $\mathbb C[G]$-module $M$ and let rank$(\alpha)$ denote the rank of $\alpha$. It will be enough to show that
$$\Delta_{\alpha}\big(\chi_G(X,f^*(\mathcal F))\big)
=\textrm{\rm rank}(\mathcal F)
\Delta_{\alpha}\big(\chi_G(X,\mathbb C)\big)-\textrm{\rm rank}(\alpha)
\textrm{\rm cond}(\mathcal F),\quad(\forall\alpha\in R(G)).$$
Note that there is a unique decomposition:
$$f_*(\mathbb C)=\bigoplus_{\alpha\in R(G)}\mathcal G_{\alpha}^
{\oplus\textrm{\rm rank}(\alpha)}$$
in the category of constructible sheaves of complex vector spaces such that for every $\alpha\in R(G)$ the rank of $\mathcal G_{\alpha}$ is equal to $\textrm{\rm rank}(\alpha)$ and the natural $\mathbb C[G]$-action on $H^0(f^{-1}(V),f^*(\mathcal G_{\alpha}))$ has character $\alpha$ where $V\subseteq Y$ is an open subcurve such that the restriction of $f$ onto $f^{-1}(V)$ is unramified. Since for every $\alpha\in R(G)$ we have:
\begin{eqnarray}
\Delta_{\alpha}\big(\chi_G(X,f^*(\mathcal F)\big)
&=&\chi(Y,\mathcal F\otimes\mathcal G_{\alpha})
,\nonumber\\
\Delta_{\alpha}\big(\chi_G(X,\mathbb C)\big)
&=&\chi(Y,\mathcal G_{\alpha}),\nonumber
\end{eqnarray}
it will be enough to show for every such $\alpha$ that
\begin{equation}
\chi(Y,\mathcal F\otimes\mathcal G_{\alpha})
=\textrm{\rm rank}(\mathcal F)
\chi(Y,\mathcal G_{\alpha})-\textrm{\rm rank}(\mathcal G_{\alpha})
\textrm{\rm cond}(\mathcal F).\end{equation}
By Theorem 3.2 for every $\alpha\in R(G)$ we have:
\begin{eqnarray}
\ \ \chi(Y,\mathcal F\otimes\mathcal G_{\alpha})
&\!\!\!\!\!=\!\!\!\!\!&\textrm{\rm rank}(\mathcal F)\textrm{\rm rank}
(\mathcal G_{\alpha})
\chi(Y,\mathbb C)-
\textrm{\rm cond}(\mathcal F\otimes\mathcal G_{\alpha}),\\
\ \ \chi(Y,\mathcal G_{\alpha})&\!\!\!\!\!=
\!\!\!\!\!&\textrm{\rm rank}(\mathcal G_{\alpha})
\chi(Y,\mathbb C)-\textrm{\rm cond}(\mathcal G_{\alpha}).
\end{eqnarray}
Since for every $x\in Y(\mathbb C)$ either $\mathcal F$ or $G_{\alpha}$ is locally constant on a neighbourhood of $x$, using Lemma 3.3 we get:
\begin{equation}
\textrm{\rm cond}(\mathcal F\otimes\mathcal G_{\alpha})=
\textrm{\rm rank}(\mathcal G_{\alpha})\textrm{\rm cond}(\mathcal F)
+\textrm{\rm rank}(\mathcal F)\textrm{\rm cond}(\mathcal G_{\alpha}),
\quad(\forall\alpha\in R(G)).
\end{equation}
Now equation (3.6.1) follows from combining (3.6.4) with equations (3.6.2) and (3.6.3).
\end{proof}
Suppose again that the map $f:X\rightarrow Y$ is the cover $\pi:\mathcal C'\rightarrow\mathcal C$ of the introduction. Let $g:\mathcal E\rightarrow\mathcal C$ and $g':\mathcal E'\rightarrow\mathcal C'$ be as above. Then $R^1g'_*(\mathbb C)$ is a complex constructible $G$-sheaf on $\mathcal C'$. Assume that the ramification divisor of the cover $\pi$ has support disjoint from the conductor of $E$ and let $c_E$ denote the degree of the conductor of $E$ as in the introduction. 
\begin{thm} We have: 
$$[H^1(\mathcal C',R^1g'_*(\mathbb C))]
=c_E[\mathbb C[G]]-2\chi_G(\mathcal C',\mathbb C).$$
\end{thm}
\begin{proof} By definition we have:
$$\chi_G(\mathcal C',R^1g'_*(\mathbb C))=
[H^0(\mathcal C',R^1g'_*(\mathbb C))]-
[H^1(\mathcal C',R^1g'_*(\mathbb C))]+
[H^2(\mathcal C',R^1g'_*(\mathbb C))].$$
As noted in the paragraph above Lemma 1.4 of \cite{CZ} on page 5 we have:
$$H^0(\mathcal C',R^1g'_*(\mathbb C))=0=
H^2(\mathcal C',R^1g'_*(\mathbb C)),$$
and hence
$$\chi_G(\mathcal C',R^1g'_*(\mathbb C))=-[H^1(\mathcal C',R^1g'_*(\mathbb C))].$$
By the proper base change theorem the complex constructible $G$-sheaves $R^1g'_*(\mathbb C)$ and $\pi^*(R^1g_*(\mathbb C))$ are isomorphic. Therefore we may use Theorem 3.6 to get
$$\chi_G(\mathcal C',R^1g'_*(\mathbb C))=\textrm{\rm rank}(R^1g_*(\mathbb C))
\chi_G(\mathcal C',\mathbb C)-
\textrm{\rm cond}(R^1g_*(\mathbb C))[\mathbb C[G]].$$
Since $\textrm{\rm rank}(R^1g_*(\mathbb C))=2$ by the proper base change theorem and $\textrm{\rm cond}(R^1g_*(\mathbb C))=c_E$, the claim is now clear.
\end{proof}

\section{Rank bounds and Hodge theory}

\begin{notn} For every smooth projective variety $X$ over $\mathbb C$ let $\textrm{NS}(X)$ be the N\'eron--Severi group of $X$ and let
$$c_1:\textrm{NS}(X)\longrightarrow H^1(X,\Omega^1_{X/\mathbb C})$$
be the Chern class map of de Rham cohomology. Let $s:\mathcal C'\rightarrow\mathcal E'$ be the zero section of the elliptic fibration $g':\mathcal E'\rightarrow\mathcal C'$ and let $T(\mathcal E')\leq\textrm{NS}(\mathcal E')$ be the subgroup generated by the algebraic equivalence classes of $s(\mathcal C')$ and the irreducible components of the fibres of $g'$. Finally let $T_{\textrm{\rm dR}}(\mathcal E')$ be the $\mathbb C$-linear span of the image of $T(\mathcal E')$ with respect to $c_1$.
\end{notn}
\begin{lemma} There is an $\mathbb C[G]$-linear injection:
$$E(F')\otimes\mathbb C\longrightarrow H^1(\mathcal E',\Omega^1_{\mathcal E'/\mathbb C})/T_{\textrm{\rm dR}}(\mathcal E').$$
\end{lemma}
\begin{proof} By the Shioda--Tate formula:
$$\textrm{NS}(\mathcal E')\otimes\mathbb C\cong E(F')\otimes\mathbb C
\oplus T(\mathcal E')\otimes\mathbb C.$$
The claim now follows from the fact that the map
$$\textrm{NS}(\mathcal E')\otimes\mathbb C\longrightarrow H^1(\mathcal E',\Omega^1_{\mathcal E'/\mathbb C})$$
induced by $c_1$ is an $\mathbb C[G]$-linear injection.
\end{proof}
\begin{prop} We have:
$$[T_{\textrm{\rm dR}}(\mathcal E')]=
[H^0(\mathcal C',R^2g'_*(\mathbb C))]+[H^2(\mathcal C',g'_*(\mathbb C))].$$
\end{prop}
\begin{proof} For every finite set $T$ let $\mathbb C[T]$ denote the $\mathbb C$-module of formal $\mathbb C$-linear combination of elements of $T$. When $T$ is equipped with a left $G$-action $\mathbb C[T]$ has a natural $\mathbb C[G]$-module structure. Let $R\subset\mathcal C'(\mathbb C)$ be the set of all points $x$ such that the fibre of $g'$ over $x$ is singular. For every $x\in\mathcal C'(\mathbb C)$ let $C_x$ denote the set of irreducible components of the fibre of $g'$ over $x$. For every $x$ as above and for every irreducible component $i\in C_x$ let $m_i$ denote the multiplicity of $i$. For every complex vector space $V$ and subset $T\subseteq V$ let $\langle T\rangle\subseteq V$ denote $\mathbb C$-linear span of $T$. Equip $\bigsqcup_{x\in R}C_x$ with the $G$-action induced by the $G$-action on $\mathcal E'$. Note that the subset:
$$M=\big\{\sum_{i\in C_x}m_ii|x\in R\big\}\subset
\mathbb C\big[\bigsqcup_{x\in R}C_x\big]$$
is $G$-invariant, therefore its $\mathbb C$-linear span is a $\mathbb C[G]$-submodule. We have the following isomorphism of $\mathbb C[G]$-modules:
$$T_{\textrm{\rm dR}}(\mathcal E')\cong T(\mathcal E')\otimes\mathbb C\cong
\mathbb C^{\oplus2}\oplus\mathbb C\big[\bigsqcup_{x\in R}C_x\big]/
\langle M\rangle,$$
where we equip $\mathbb C$ with the trivial $\mathbb C[G]$-module structure. Because the fibres of $g'$ are connected, for every connected open subset $V\subseteq\mathcal C'$ we have $H^0((g')^{-1}(V),\mathbb C)=\mathbb C$, and hence the sheaf $g'_*(\mathbb C)$ is constant of rank $1$. Consequently $H^2(\mathcal C',g'_*(\mathbb C))$ is isomorphic to the trivial $\mathbb C[G]$-module of dimension one.

Because the map $g:\mathcal E\rightarrow\mathcal C$ is projective, there is a closed immersion $l:\mathcal E\rightarrow\mathbb P^n_{\mathcal C}$ of $\mathcal C$-schemes for some $n\in\mathbb N$ where $p:\mathbb P^n_{\mathcal C} \rightarrow\mathcal C$ is the projective $n$-space over $\mathcal C$. The base change $l':\mathcal E'\rightarrow\mathbb P^n_{\mathcal C'}$ of $l$ with respect to $\pi$ is a $G$-equivariant closed immersion of $\mathcal C'$-schemes where we equip the $\mathcal C'$-scheme $\mathbb P^n_{\mathcal C'}$ with the left $G$-action induced by the natural isomorphism between $\mathbb P^n_{\mathcal C'}$ and the base change of $\mathbb P^n_{\mathcal C}$ to $\mathcal C'$. With respect to this action the structure map $p':\mathbb P^n_{\mathcal C'}\rightarrow\mathcal C'$ is $G$-equivariant. The map $l'$ furnishes a $G$-equivariant $\mathbb C$-linear homomorphism $(l')^*:R^2p'_*(\mathbb C)\rightarrow R^2g'_*(\mathbb C)$ of complex constructible $G$-sheaves. By the K\"unneth formula $R^2p'_*(\mathbb C)=\mathbb C$. For every $x\in\mathcal C'(\mathbb C)$ let $(l')^*|_x:R^2p'_*(\mathbb C)|_x\rightarrow R^2g'_*(\mathbb C)|_x$  denote the fibre of $(l')^*$ at $x$. By the proper base change theorem there is a $\mathbb C$-linear commutative diagram:
$$\begin{CD}
R^2p'_*(\mathbb C)|_x@>(l')^*|_x>>R^2g'_*(\mathbb C)|_x\\
@VVV@VVV\\
\mathbb C@>d_x>>\mathbb C[C_x].\\
\end{CD}$$
such that both vertical arrows are isomorphisms and $d_x$ has image $\langle\sum_{i\in C_x}m_ii\rangle$. For every $x\in\mathcal C'(\mathbb C)$ let $i_x:x\rightarrow\mathcal C'$ denote the closed immersion of the point $x$ into $\mathcal C'$. Note that for every complex vector space $V$ the direct image $(i_x)_*(V)$ is a skypsaker sheaf on $\mathcal C'$.
By the above there is a $\mathbb C$-linear and $G$-equivariant short exact sequence:
\begin{equation}
\CD0@>>>\mathbb C@>>>R^2g'_*(\mathbb C)
@>>>\mathcal F@>>> 0\endCD
\end{equation}
of complex constructible $G$-sheaves where
$$\mathcal F=\bigoplus_{x\in R}(i_x)_*\big(\mathbb C[C_x]/
\langle\sum_{i\in C_x}m_ii\rangle\big),$$
equipped with the tautological $G$-action. Note that every $x\in\mathcal C'(\mathbb C)$ has a contractible open neighbourhood $V\subset\mathcal C'$ such that the first cohomology of the constant sheaf $\mathbb C$ on $V$ vanishes. Therefore the restriction of the short exact sequence (4.3.1) onto $V$ splits. Because the sheaf $\mathcal F$ has finite support the short exact sequence (4.3.1) splits on $\mathcal C'$, too. Therefore the sequence
$$\CD0@>>>H^0(\mathcal C',\mathbb C)@>>>H^0(\mathcal C',R^2g'_*(\mathbb C))
@>>>H^0(\mathcal C',\mathcal F)@>>> 0\endCD$$
is also exact. We get that
$$[H^0(\mathcal C',R^2g'_*(\mathbb C))]=
[\mathbb C]+[\mathbb C\big[\bigsqcup_{x\in R}C_x\big]/\langle M\rangle].$$
and the claim follows.
\end{proof}
\begin{thm} We have:
$$[H^1(\mathcal E',\Omega^1_{\mathcal E'/\mathbb C})/T_{\textrm{\rm dR}}(\mathcal E')]=(c_E-d_E/6)[\mathbb C[G]]-\chi_G(\mathcal C',\mathbb C).$$
\end{thm}
\begin{proof} The $G$-equivariant degenerating spectral sequence:
$$H^q(\mathcal E',\Omega_{\mathcal E'/\mathbb C}^p)
\Rightarrow H^{p+q}(\mathcal E',\mathbb C)$$
furnishes the equation:
\begin{equation}
[H^2(\mathcal E',\mathbb C)]=
[H^0(\mathcal E',\Omega^2_{\mathcal E'/\mathbb C})]+
[H^1(\mathcal E',\Omega^1_{\mathcal E'/\mathbb C})]+
[H^2(\mathcal E',\mathcal O_{\mathcal E'})].
\end{equation}
By Lemma 1.4 of \cite{CZ} on page 5 the $G$-equivariant spectral sequence:
\begin{equation}
H^p(\mathcal C',R^qg'_*(\mathbb C))
\Rightarrow H^{p+q}(\mathcal E',\mathbb C)\nonumber
\end{equation}
degenerates, so we get that
\begin{equation}[H^2(\mathcal E',\mathbb C)]=
[H^0(\mathcal C',R^2g'_*(\mathbb C))]+
[H^1(\mathcal C',R^1g'_*(\mathbb C))]+
[H^2(\mathcal C',R^0g'_*(\mathbb C))].\end{equation}
Combining equations (4.4.1) and (4.4.2) with Proposition 4.3 we get that
\begin{eqnarray}
[H^1(\mathcal E',\Omega^1_{\mathcal E'/\mathbb C})/
T_{\textrm{\rm dR}}(\mathcal E')]&
\!\!\!\!\!=\!\!\!\!\!&
[H^1(\mathcal C',R^1g'_*(\mathbb C))]
-[H^0(\mathcal E',\Omega^2_{\mathcal E'/\mathbb C})]\\
&&\!\!\!\!\!-[H^2(\mathcal E',\mathcal O_{\mathcal E'})].\nonumber
\end{eqnarray}
Let $(\cdot)^{\vee}:K(\mathbb C[G])\rightarrow K(\mathbb C[G])$ denote the unique group homomorphism such that $[M]^{\vee}$ is the isomorphism class of the dual Hom${}_{\mathbb C}(M,\mathbb C)$ for every finitely generated $k[G]$-module $M$. Serre's duality furnishes a perfect pairing
\begin{equation}
H^0(\mathcal E',\Omega^2_{\mathcal E'/\mathbb C})\times H^2(\mathcal E',\mathcal O_{\mathcal E'})
\rightarrow
H^2(\mathcal E',\Omega^2_{\mathcal E'/\mathbb C})=\mathbb C\end{equation}
which is $\mathbb C[G]$-linear, therefore
\begin{equation}
[H^0(\mathcal E',\Omega^2_{\mathcal E'/\mathbb C})]=
[H^2(\mathcal E',\mathcal O_{\mathcal E'})]^{\vee}=\frac{d_E}{12}[\mathbb C[G]]-\chi_G(\mathcal C',\mathcal O_{\mathcal C'})^{\vee}
\end{equation}
by Theorem 2.5. Because the boundary maps in the spectral sequence:
$$H^q(\mathcal C',\Omega_{\mathcal C'/\mathbb C}^p)
\Rightarrow H^{p+q}(\mathcal C',\mathbb C)$$
are $\mathbb C[G]$-linear we get that
\begin{equation}
\chi_G(\mathcal C',\mathbb C)=
\chi_G(\mathcal C',\mathcal O_{\mathcal C'})-\chi_G(\mathcal C',\Omega^1_{\mathcal C'/\mathbb C}).
\end{equation}
Serre's duality furnishes a perfect pairing
$$
H^0(\mathcal C',\Omega^1_{\mathcal C'/\mathbb C})\times H^1(\mathcal C',\mathcal O_{\mathcal C'})
\rightarrow
H^1(\mathcal C',\Omega^1_{\mathcal C'/\mathbb C})=\mathbb C
$$
which is $\mathbb C[G]$-linear, therefore
\begin{equation}
\chi_G(\mathcal C',\Omega^1_{\mathcal C'/\mathbb C})=-\chi_G(\mathcal C',\mathcal O_{\mathcal C'})^{\vee}.
\end{equation}
Combining Theorem 2.5 and (4.4.5) with equations (4.4.6) and (4.4.7) we get that
\begin{equation}
[H^0(\mathcal E',\Omega^2_{\mathcal E'/\mathbb C})]+
[H^2(\mathcal E',\mathcal O_{\mathcal E'})]=\frac{d_E}{6}[\mathbb C[G]]-\chi_G(\mathcal C',\mathbb C).
\end{equation}
The claim now follows from equations (4.4.3) and (4.4.8) and from Theorem 3.7.
\end{proof}
\begin{proof}[Proof of Theorem 1.1] We are going to use the notation of the proof of Theorem 3.6. By Lemma 4.2 it will be enough to show that
$$\Delta_{\alpha}\big(
[H^1(\mathcal E',\Omega^1_{\mathcal E'/\mathbb C})/
T_{\textrm{\rm dR}}(\mathcal E')]
\big)\leq\textrm{\rm rank}(\alpha)(c_E-d_E/6+2g-2+\deg(S))$$
for every $\alpha\in R(G)$. In order to do so, it will be enough to prove that
$$-\Delta_{\alpha}\big(\chi_G(\mathcal C',\mathbb C)\big)\leq
\textrm{\rm rank}(\alpha)(2g-2+\deg(S))$$
for every $\alpha\in R(G)$ by Theorem 4.4. We have
$$-\Delta_{\alpha}\big(
\chi_G(\mathcal C',\mathbb C)\big)=-\chi
(\mathcal C,\mathcal G_{\alpha})=\textrm{\rm rank}(\mathcal G_{\alpha})(2g-2)+
\textrm{\rm cond}(\mathcal G_{\alpha})$$
for every $\alpha\in R(G)$ by Theorem 3.2. Also note that
$c_x(\mathcal G_{\alpha})\leq\textrm{\rm rank}(\mathcal G_{\alpha})$
for every $x\in S$ because $\dim_{\mathbb C}\big((\mathcal G_{\alpha})_x\big)$ is non-negative. Since $c_x(\mathcal G_{\alpha})=0$ for every $x\in X(\mathbb C)-S$, the claim is now clear.
\end{proof}

\end{document}